\documentclass[12pt]{amsart}
\usepackage{amsmath,amsfonts,amsthm,amsopn,cite,mathrsfs}
\usepackage{epsfig,verbatim}
\usepackage{subfigure}
\newcommand{\gaul}{(\bZ +i\bZ )^d}
\newcommand{\duall}{\Lambda ^{\circ}}

\newcommand{\eps}{\epsilon}
\newcommand{\ccG}{\mathcal G}
\newcommand{\BF}{Bargmann-Fock space}
\newcommand{\bfd}{\cF _d}
\newcommand{\bC}{\mathbb C}
\newcommand{\bR}{\mathbb R}
\newcommand{\bZ}{\mathbb Z}
\newcommand{\bN}{\mathbb N}

\newcommand{\cd}{{\mathbb C}^d}
\def\intrd{\int_{\rd}}

\def\intrd{\int_{\rd}}

\newcommand{\cF}{\mathcal F}
\newcommand{\cE}{\mathcal E}
\newcommand{\cD}{\mathcal D}
\newcommand{\cZ}{\mathcal Z}
\newcommand{\cG}{\mathcal G}

\newcommand{\rd}{\bR^d}
\newcommand{\rdd}{\bR^{2d}}
\newcommand{\inv}{^{-1}}
\newcommand{\rems}{Rmarks}
\newcommand{\bba}{\mathbf a}
\newcommand{\bbb}{\mathbf b}
\newcommand{\mbZ}{\,\mathbb Z [i]}
\newcommand{\fif}{if and only if}
\def\lrd{L^2(\rd)}

\newcommand{\vs}{\bigskip}

\newtheorem{tm}{Theorem}[section]    
\newtheorem{lemma}[tm]{Lemma}
\newtheorem{prop}[tm]{Proposition}
\newtheorem{cor}[tm]{Corollary}

\newcommand{\tfs}{time-frequency shift}

\newcommand{\beq}{\begin{equation}}
\newcommand{\eeq}{\end{equation}}

\begin{document}
\begin{abstract}
We give a general construction of entire functions in $d$ complex
variables that vanish on a lattice of the form $\Lambda = A (\bZ +
i\bZ )^d$ for an invertible complex-valued matrix. As an application
we exhibit a class of lattices of density $>1$ that fail to be a sampling set for the
Bargmann-Fock space in $\bC ^2$. By using an equivalent real-variable
formulation, we show that these lattices fail to generate a Gabor
frame. 
\end{abstract}

\title[Lattice Sampling of Entire Functions]{
Sampling of Entire Functions of Several Complex Variables on a Lattice and 
  Multivariate Gabor Frames }
\author{Karlheinz Gr\"ochenig}
\address{Faculty of Mathematics \\
University of Vienna \\
Oskar-Morgenstern-Platz 1\\
A-1090 Vienna, Austria}
\email{karlheinz.groechenig@univie.ac.at}
 \author{Yurii Lyubarskii}
\address{Department of Mathematical Sciences, Norwegian University of
 Science and Technology, NO-7491, Trondheim, Norway} 
\email{yura@math.ntnu.no} 
\date{}
\thanks{K.\ G.\ was
  supported in part by the  project 31887-N32 the
Austrian Science Fund (FWF)}
\maketitle

\section{Introduction}

We study the sampling problem in the Bargmann-Fock space of several
complex variables and the related construction of Gabor frames with a
Gaussian window. Our main point is the restriction to sampling on a
lattice in $\cd$ and the consequences resulting from  the additional
invariance properties. 

The first problem is a sampling problem for entire functions in
several complex variables.  Recall that 
the Bargmann-Fock space $\cF ^2 _d$ consists of all 
entire functions of $d$ complex variables $z= (z_1, \dots, z_d) \in
\cd $ with finite norm  
\begin{equation}
  \label{eq:n1}
\| F\|_{\cF^2_d}^2 =  \int _{\cd } |F(z)|^2 e^{-\pi |z|^2} \, dz . 
\end{equation}
A set  $\Lambda \subseteq \cd $ is called  a {\em sampling set} for
$\cF ^2_d$, if, for some constants $A,B>0$, 
\beq
\label{eq:n2}
A \|F\|_{\cF ^2_d}^2 \leq \sum_{\lambda \in \Lambda} |F(\lambda)|^2 e^{-\pi |\lambda|^2}
 \leq B \|F\|_{\cF _d^2}^2, \quad \forall F\in \cF ^2_d. 
\eeq

Our second question deals with the spanning properties of \tfs s of
the 
 Gaussian function $\phi(x)=
 \exp(-\pi |x|^2)$, 
 $x\in\rd$, and thus is a  problem for functions on $\rd $.  Let   $\lambda=\xi+i \eta \in \cd$ with  $\xi,\eta \in \rd$   and 
$$
(\pi_\lambda \phi)(x)=
e^{2i\pi <\eta, x>}e^{- \pi |x-\xi|^2}
$$
be the  corresponding \tfs\ of $\phi $ by $\lambda $. 
Given a discrete set $\Lambda \subseteq \cd $, we denote the set of
\tfs s along $\Lambda $ by 
$$
\ccG(\Lambda)= \{\pi_\lambda \phi : \lambda\in \Lambda \} \, ,
$$
which is usually  called a  Gabor family.   We then say that  $\ccG(\Lambda)$ is {\em a frame} for   
$L^2(\rd)$,  if  for some constants $A  , B>0$
\beq
A \|f\|^2_{L^2(\rd)}  \leq \sum_{\lambda\in \Lambda} |\langle f, \pi_\lambda \phi \rangle |^2
\leq B  \|f\|^2_{L^2(\rd)}, \quad \forall  f \in  L^2(\rd).
\eeq
We are interested in the frame property of $\ccG(\Lambda)$ in
$L^2(\rd)$. This is a ``real-variable'' problem about functions in
$\lrd $. 

It is well-known that the Bargmann transform maps the \tfs s $\pi
_\lambda \phi$ to the normalized reproducing kernel of $\cF
^2_d$~\cite{bargmann61}. Therefore these two problems are equivalent
via the Bargmann transform:
 
 {\em The system $\cG(\Lambda)$  forms a frame in $L^2(\bR^d)$,  if
   and only if  $\Lambda$ is sampling in $\cF^2_d$.}
  See Section~\ref{duality1} for a detailed description.

 In  the case of one variable  $d=1$ the sampling property in $\cF
 ^2_1$ and  frame 
property of $\cG (\Lambda )$ can be completely characterized  in  terms of the 
density of $\Lambda $  by the results in \cite{L1, S1,
  S2}. Precisely, a separated
set $\Lambda \subseteq \bC $ is a sampling set for $\cF ^2_1$, \fif\
its  (lower) Beurling $D^-(\Lambda )$ is greater than $ 1$. See  also
the discussion in \cite{G2,G4,He}.
In the multivariate case  the density condition $D^-(\Lambda ) >1$  is 
necessary (see \cite{OS,MP00}), but is far from sufficient. Sufficient conditions in terms of
a covering density are given in~\cite{fg89jfa,MP00}, but they imply a
large  Beurling density. 

For more detailed results,  additional arithmetic conditions are
required. 
For this reason we restrict our attention exclusively to  lattices. 
As in \cite{G1} we consider only complex lattices.  By a complex
lattice we understand 
a   lattice of the form
$$\Lambda = A  (\bZ + i \bZ )^d = A \mbZ ^d
$$
for some $A\in \mathrm{GL}\, (d,\bC )$. Throughout  we will write $\mbZ = \bZ +
i\bZ $ for the ring of Gaussian integers in $\bC$. 
Since $\Lambda $ is a discrete subgroup of $\cd $, the corresponding
sampling set and the Gabor family $\cG (\Lambda )$ possess 
 an additional   structure.

To answer the   question about  Gabor frames over a lattice,
one applies the  fundamental
duality theory~\cite{feichtinger-kozek98,janssen95,ron-shen97} (see also
Section~\ref{duality1} below).  One of the key points is relation between the sampling problem in the space 
$\cF^2_d$ and the uniqueness problem in the space $\cF^\infty_d$,
which  consists of entire functions $F$ on 
$ \cd$ such that 
$$
\|F\|_{\cF^\infty_d}= \sup_{z \in \cd} |F(z)| \, e^{-\frac {\pi |z|^2}2} < \infty.
$$
After a suitable reformulation, we see that the
construction of Gabor frames is  intimately connected to two fundamental problems about
entire functions. 
 

 
 \begin{itemize}
 \item[(i)] Construct  an entire  function $\sigma _\Lambda  $ with possibly smallest growth that vanishes on
   $\Lambda $, which in analogy to the one-dimensional case we call a
   sigma-type function for $\Lambda $.
 \item[(ii)] Construct an entire function $\tau _\Lambda $ that is interpolating
   on $\Lambda $, i.e., $\tau _\Lambda (\lambda ) = \delta _{\lambda
     ,0}$ for $\lambda \in \Lambda $.  
 \end{itemize}

 In dimension $d=1$ the above problems on $\mbZ $  are solved by the
 classical Weierstrass  $\sigma$-function  
$$
\sigma (z)=z\prod_{m,n\in \bZ, \  (m,n)\neq (0,0)}\left ( 1-\frac z {m+in} \right )
e^{\frac z {m+in} +\frac 1 2\frac  {z^2}{(m+in)^2}} \, , \quad \,\, \tau(z)=\frac {\sigma(z)}z \, .
$$
Clearly, $\sigma $  is an entire function   with $\mbZ$
as its  zero set, and $\tau $ in interpolating on $\mbZ $. 

By contrast, in  several complex variables the above
questions are rather unusual.   The zero set of an
entire function of $d>1$ variables is always an analytic
manifold. Although it is possible to extend the construction of
the Weierstrass product  to obtain entire functions whose zero set is
a given analytic hypersurface ~\cite{GL86,ronkin}, this
construction sheds no light 
on the search of  sigma-type functions on a lattice. Only few results about
interpolation with discrete sets are known, see, e.g.,~\cite{pap87,pap92}.

Our first contribution is a general recipe for the  construction of   sigma-type functions and interpolating functions
associated to an arbitrary complex lattice. The idea  builds essentially on the
one-dimensional machinery and  yields  a  special class of sigma-type  functions, whose
zero set  is a union of analytic planes that contain the original
lattice.   Entire
functions on $\bC ^2$ with plane zeros also  play an important role in
~\cite{LR00}. In principle this idea works for arbitrary dimensions, but
we will restrict ourselves to entire functions of \emph{two}
variables.

Our second contribution is the application of the general construction
of sigma-type functions 
to show that certain complex lattices 
fail to yield sampling sets for $\cF ^2_2$.  This requires the
control  of    the growth of the
sigma-type function so that it is in $ \cF ^2_d$.  For the
application to sampling in Bargmann-Fock space the goal is therefore
to find sigma-type functions with small growth. 

To provide an idea of the main  construction, we consider a model example 
that inspired our general construction.

Let 
\begin{equation}
 \label{matrixaex}
A=\begin{pmatrix}
  1 & 1/2 \\
0 & \sqrt{3}/2
\end{pmatrix},  \quad   \Lambda=A   \mbZ^2.
\end{equation}
This lattice is the complexification of the usual hexagonal lattice in
$\bR ^2$. 
Our main construction then suggests the following sigma-type function
for $\Lambda $: 
\begin{equation}
  \label{eq:n1}
\sigma _\Lambda (z_1,z_2)= \sigma(z_1) \, \sigma (\frac{z_2}{\sqrt 3} -
\frac{1}{2} ) \, \sigma (\frac{z_2}{\sqrt 3} -
\frac{i}{2}  ) \,  \sigma (\frac{z_2}{\sqrt 3} -
\frac{1+i}{2}  ) e^{2\pi (1-i) z_2} \, .
\end{equation}

It is easy to check that    $\sigma _\Lambda $ vanishes on $\Lambda $ and satisfies the growth estimate
  $|\sigma _\Lambda (z)| \leq C e^{\pi |z|^2/2}$.

The general characterization of lattice sampling sets  
(Proposition~\ref{chargauss}) then  implies the following result: 
\begin{center}
\emph{$\Lambda $ fails to be a sampling set for $\cF ^2_2$.}  
\end{center}



By contrast,  the  sigma-type function
$$
\sigma _2  (A\inv z) = \sigma  (z_1-\frac{z_2}{\sqrt{3}}) \sigma (\frac{2z_2}{\sqrt{3}})
$$
also vanishes on $\Lambda $, but it grows much too fast to be of use in the analysis of $\cF ^2_2$.

A small modification of \eqref{eq:n1} yields the interpolating function
\begin{equation}
  \label{eq:n2}
\tau _\Lambda  (z_1,z_2)= \frac{\sigma(z_1)}{z_1} \, \sigma (\frac{z_2}{\sqrt 3} -
\frac{1}{6} ) \, \sigma (\frac{z_2}{\sqrt 3} -
\frac{i}{6}  ) \sigma)(\frac{z_2}{\sqrt 3} -
\frac{1+i}{6}  ) e^{\pi (1-i)/\sqrt{3} z_2} \, .
  \end{equation}
 Although $\tau _\Lambda  \not \in \cF _2^2$, it satisfies the growth
 rate $\phi _\Lambda (z) \leq C e ^{\pi |z|^2/2}$. This suffices to
 derive a weak   Lagrange interpolation formula. 
See Theorem~\ref{weaksamp}. 

This example is rather puzzling. In dimension $d=1$ the hexagonal
lattice $A\bZ ^2$  has density $2/\sqrt{3}>1$ and generates a Gabor
frame $\cG (A\bZ ^2)$ with certain  optimal
features~\cite{SB03,FS17}.
By contrast in dimension $d=2$, $A\mbZ ^2$  fails to generate a Gabor
frame and a sampling set for $\cF ^2_2$, although it has density $4/3>1$.

Currently the investigation of sampling the Bargmann-Fock space on
lattices is poorly understood, and still amounts to the investigation
of examples and counter-examples. We hope that this article will stir
some interest among  the experts in several complex variables and that it
will inspire a deeper analysis of the problem. 

 \medskip

The article is organized as follows.  In Section~2 we collect the
background material about the connection between sampling in
Bargmann-Fock space and Gabor frames, the basic information about the
Weierstrass sigma-function, and a normalized representation of
lattices by means of Minkowski reduced bases. Section~3 contains the
main construction of sigma-type functions and interpolating functions
for complex lattices. In Section~4 we use this class of sigma-type
functions to show that certain ``natural'' lattices of density $>1$
fail to be sampling in Bargmann-Fock space. In Section~5 we prove a
weak Lagrange interpolation formula for certain lattices of density
$\geq 1$.


\section{Sampling, Sigma Functions, and Lattices}

\subsection{Sampling in Bargmann-Fock space and Gabor frames }     
\label{duality1}
The lattice structure  leads to special
criteria  for a set $\Lambda $ to be sampling. We emphasize that these are unavailable
for arbitrary sets of points. 


 

The  relation between the sampling property of the lattice $\Lambda$
and the frame property of the system $\ccG(\Lambda)$ 
 is  summarized in the following statement. For its formulation we
 need the  adjoint lattice $\duall =(A^*)\inv
\mbZ^d $  of $\Lambda = A \mbZ ^d$ 

\begin{prop} \label{chargauss}
 
For a  lattice $\Lambda = A \mbZ ^d  \subseteq \bC ^d$  the following are
equivalent: 

(i) $\ccG(\Lambda)$ is a frame in $L^2(\rd)$;

(ii) $\Lambda $ is a  sampling set for $\cF^2_d$; 

(iii)  There exists an interpolating  function $G \in \cF^2_d $ for
$\Lambda ^\circ$  satisfying  the Bessel property, 
i.e.,  
\begin{equation}
  \label{eq:22}
  G(\mu ) = \delta _{\mu ,0} , \qquad \text{ for all } \mu \in \Lambda
  ^\circ \, , 
\end{equation}
and  $F \to \big( \langle F , e^{\pi \bar{\lambda } \cdot z} G(z − λ)
\rangle e^{-\pi |\lambda |^2/2}\big)_{\lambda \in \Lambda } $ maps
$\cF ^2_d$ to
$\ell ^2(\Lambda )$.

(iv) $\Lambda $ is a set of uniqueness for $\cF ^\infty_d $, i.e., if $F
\in \cF ^\infty_d $ and $F (\lambda ) = 0$ for all $\lambda \in \Lambda
$, then $F \equiv 0$. 
\end{prop}

Proposition~\ref{chargauss} explains the fundamental importance  of  sigma-type
functions and interpolating functions in $\cF ^2_d$  for the theory of
Gabor frames. 
It follows from this proposition that a function $F\in \cF ^2_d$ can 
be recovered from its samples $\{F(\lambda)\}_{\lambda\in \Lambda}$
by mean of the interpolating function $G$  constructed   for $\duall $, not for $\Lambda
$!  Namely \footnote{We always  use the "real" inner product: $z\cdot
  w=\sum _{j=1}^d z_j w_j$ for
$z=(z_j), \ w=(w_j) \in \bC ^d$. }
\begin{equation}
  \label{eq:n5}
  F(z) = \sum _{\lambda \in \Lambda } F(\lambda ) e^{\pi \bar{\lambda
    }\cdot z} G(z-\lambda ) e^{-\pi |\lambda |^2/2} \, .
\end{equation}


The connection (i) and (ii) is classical. Let $f\in \lrd $ and 
\begin{equation}
  \label{eq:n9}
  Bf(z) = 2^{d/4} \, e^{-\pi z\cdot z /2} \, \intrd f(x) e^{-\pi |x|^2 + 2\pi x\cdot z} \, dx 
 \end{equation}
be the Bargmann transform of $f$. Then $B$ is unitary from $\lrd $
onto the Bargmann-Fock space $\cF ^2_d$, e.g., by
\cite{folland89}. Moreover, 
 the Bargmann transform maps
the \tfs\ $2^{d/4}\, \pi _\lambda \phi $ to the normalized  reproducing kernel $e^{\pi
  \bar{\lambda } \cdot z} e^{-\pi |\lambda |^2/2}$ of the Fock
space. Therefore 
\begin{equation}
  \label{eq:n11}
  \langle f, \pi _\lambda \phi \rangle _{L^2} =
\langle Bf, B(\pi _\lambda \phi )\rangle _{\cF ^2_d} = Bf(\lambda )
e^{-\pi |\lambda |^2/2} \, .
\end{equation}
See ~\cite{folland89,G2} and the original
literature~\cite{bargmann61}.

The equivalence of (ii) and (iii) seems beyond the realm of complex
analysis and is due to the invariance properties of lattices. It shows
that the problem of sampling on a given lattice $\Lambda $ is
equivalent to an interpolation problem of the adjoint lattice $\Lambda
^\circ $. This statement  is  part of the duality theory of
Gabor frames~\cite{feichtinger-kozek98,janssen95,ron-shen97}. 

The
equivalence of (ii) and (iv)  follows from  one of  the characterizations of
Gabor frames without 
inequalitites~\cite{G3} via the Bargmann transform. 
\medskip

\subsection{Sigma-type functions for lattices}


 Our  main  tool for the construction of sigma-type functions in $\cd
 $  is  the classical Weierstrass  $\sigma$-function  of {\em one} variable 
$$
\sigma (z)=z\prod_{\lambda \in \mbZ }\left ( 1-\frac z {\lambda }
\right ) e^{\frac z {\lambda } +\frac 1 2\frac  {z^2}{\lambda ^2}}.
$$
We refer to \cite{A1,Simon2} for basic  properties of $\sigma(z)$. In
particular, $\sigma$  is an entire function   with $\mbZ$ as the zero set.
 
In addition , see e.g. \cite{H1}, for each $\eps>0$
\begin{equation}
\label{hayman}
| \sigma(z)| \asymp e^{\pi |z|^2/2}, \quad \mathrm{dist}(z,\mbZ)>\eps.
\end{equation}
We note that $\sigma \in \cF ^\infty _1 \setminus \cF ^2_1$. By
Proposition~\ref{chargauss} the lattice of Gaussian integers $\mbZ $
fails to be a sampling set for $\mbZ $. This fact was already proved
in~\cite{bargmann71,perelomov71}

For $z=(z_1, \ldots \ , z_d)\in \cd$ we set
\begin{equation}
  \label{eq:c1}
  \sigma _0 (z) = \prod _{j=1}^d \sigma (z_j) \, .
\end{equation}
Then the function 
\begin{equation}
  \label{eq:c2}
  \sigma _A (z) = \sigma _0(A\inv z)
\end{equation}
vanishes on $\Lambda $ and satisfies the growth estimate 
 \begin{equation}
 \label{eq:c2a}
  |\sigma _A  (z) | \leq C e^{\pi \|A\inv \|_{op}^2 |z|^2/2}\, ,
\end{equation}
where as usual $\|A\|_{op}$ denotes the largest singular value of $A$.

Similarly we construct the {\em interpolating } function for $\Lambda=
A \mathbf{Z}^d$. 
 We define $\tau(z)=z^{-1}\sigma(z)$ for $z\in \bC$ and
 \begin{equation}
 \label{eq:c2c}
 \tau_0(z)=\prod_{j=1}^d \tau (z_j); \ \tau_\Lambda(z)=\tau_0(A^{-1}z) \ \mbox{for}  \  z=(z_1,\ldots \ , z_d)\in \bC^d. 
 \end{equation}
For example, consider  the lattice $\alpha \mbZ  \subseteq \bC $ with
adjoint lattice $\duall = \alpha \inv \mbZ $. Then $\duall $
possesses the interpolating function $\tau (\alpha z)$, which belongs
to $\cF ^2_1$, \fif\ $0<|\alpha | <1$. By Proposition~\ref{chargauss}
$\alpha \mbZ  $ is a sampling set for $\cF ^2_1 $, \fif\ $|\alpha | <1$. 
This is a baby version of  the complete characterization of
one-dimensional sampling sets in~\cite{L1, S1,S2}.  
\medskip


\subsection{Lattice reduction and Minkowski-reduced bases}

Clearly, the functions $\sigma _\Lambda $ and $\tau _\Lambda $ depend
on the generating matrix $A$, or equivalently, on the choice of a
basis of $\Lambda $. Our first task is to choose a suitable basis of
$\Lambda $, so that $\|A\inv \|_{\mathrm{op}}$ is small. 
One of possible recipes is to choose
a basis of $\Lambda $ which consists of the shortest possible  vectors. 
 
Precisely, choose  vectors $\mathbf{a}_j\in \Lambda$,
such that $\|\bba _1 \| = \min \{\|\lambda \| : \lambda \in \Lambda
\}$ and 
$$
\|\bba _j \| = \min \{ \|\lambda \| : \lambda \in \Lambda , \lambda
\not \in \mathrm{span}\, [\bba _1, \dots \bba _{ j-1}] \} \, .
$$
Such a  basis is called 
Minkowski reduced for $\Lambda $ ~\cite{LLL}, and  satisfies
$$
\|\bba _1\| \leq \| \bba _2 \| \leq \dots \leq  \| \bba _d \| \, . 
$$
After setting $A = \big(  \mathbf{a} _1 \,\,  \mathbf{a} _2 \,\, \dots
\,\, \mathbf{a} _d \big)$,  
the lattice is $\Lambda = A \mbZ ^d$. 

 We may  write $A=US$, where $U$ is a
 unitary matrix   and $S$ is an upper triangular matrix with real
 values on the diagonal and 
 with columns $\bbb _j$ (by  QR-factorization or  Gram-Schmidt
 orthogonalization). Since the Fock space $\cF ^2_d$ is invariant
 under a unitary transformation  $U$  of coordinates, a set $\Lambda $
 is sampling, \fif\ $U\inv \Lambda $ is sampling. Replacing $\Lambda =
 A\mbZ ^d$ by $U\inv \Lambda = S \mbZ ^d$, we may therefore
 assume without loss of generality that  $\Lambda = S \mbZ ^d$ for an
 upper triangular matrix $S$. 
 
Since $U$ preserves lengths, the columns
of $S$ are also Minkowski-reduced basis, and  $\|\bbb _1\| \leq \| \bbb _2 \|
\leq \dots \leq \| \bbb _d \| $.  The upper triangular matrix $S$ is of
the form 
\begin{equation}
  \label{eq:r1}
S = (s_{jk})= 
\begin{pmatrix}
   \gamma _1 & \ast & \dots & \ast    \\
0 & \gamma  _2 &   \dots & \ast  \\
\vdots \\
 0 &  \dots  & 0 & \gamma _d 
\end{pmatrix} \, ,
\end{equation}
and its entries satisfy the additional conditions 
\begin{align}
s_{j,k}  &= 0 \qquad k<j, \quad s_{kk}  = \gamma _k
  >0,       \label{eq:n3} \\
 |\mathrm{Re} & \, s_{j,k} |, \quad |\mathrm{Im}\, s_{j,k} |   \leq
  \frac{s_{j,j}}{2} \, . \end{align}
 In   \cite{G1} we proved the following result for general
 upper-diagonal matrices in arbitrary dimension.  

\begin{prop}
  \label{known}
  Let $\Lambda = S\mbZ ^d $ with an upper triangular matrix $S$. 

(i)   If $\gamma _j <1$ for $j=1, \dots ,d$, then $\Lambda $ is a sampling
  set for $\cF ^2_d$ (and $\cG (\Lambda )$ is a frame for $\lrd $).

 (ii) If $\gamma_d\geq 1$, then $\Lambda $ is not sampling and
 $\ccG(\Lambda)$ is not  a frame.  
\end{prop}

The proof goes roughly as follows. The adjoint lattice is  $\Lambda ^\circ
= (S^*)\inv \mbZ ^d$, where $(S^*)\inv $ is a lower trigonal matrix
with diagonal $ (\lambda _j \inv )_{j=1, \dots ,d}$. 
Then  the function  
$$
F_0(z_1,z_2, \dots , z_d)= \prod _{j=1}^d \frac {\sigma(\gamma_j z_j)}{z_j}
$$
belongs to $\cF^2_d$ and solves the interpolation problem $(iii)$ in Proposition \ref{chargauss}.
If $\gamma _d\geq 1$, then the  function $F_1(z_1,z_2, \dots , z_d):=
\sigma(\gamma_d^{-1} z_d)$ belongs to $\cF^\infty _d$ and vanishes on
$\Lambda$. By Proposition \ref{chargauss} $\Lambda $ fails to be a
sampling set. 

A different class of examples is discussed in ~\cite{PR13}.

From now on we deal with explicit constructions in two complex
variables.  We may assume without loss of
generality that $\Lambda $ is determined by the matrix
\begin{equation}
\label{matrixa}
A =
\begin{pmatrix}
  \gamma _1 & \beta \\
0 & \gamma _2
\end{pmatrix}
\end{equation}
with 
\begin{equation}
\label{eq:d9a}
\gamma _1, \gamma _2 >0, \quad  \gamma _1 ^2 \leq |\beta |^2 + \gamma _2^2. 
\end{equation}
Since the basis is reduced, we  have 
\begin{equation}\label{eq:d9b}
  |\mathrm{Re}\, \beta|, \quad |\mathrm{Im}\, \beta|   \leq \frac{\gamma
  _1}{2} .
\end{equation}


\section{A  Construction of Sigma-Type Functions and Interpolating
  Functions via Sublattices } 

To go beyond Proposition~\ref{known}, we need a more involved
recipe for sigma-type and interpolating functions. 
In  this section we provide such  a  general construction  
inspired  by the model 
example \eqref{eq:n1}. The goal is to produce sigma-type  functions or
interpolation functions   with   small   
growth. 

Given a lattice   $\Lambda = A \mbZ^2 $  in $\bC^2 $, we will use the
following master plan to construct a   sigma-type  function 
 function or an interpolating function  on $\Lambda $ via a sublattice. 
\begin{enumerate}
\item Construct a sublattice $\Gamma \subseteq \Lambda  $ that
  possesses an orthogonal basis (or possibly a nearly orthogonal
  basis). 
\item Construct a sigma function $\sigma _\Gamma $  and an
  interpolating function $\tau _\Gamma $ on
  $\Gamma $ according to~\eqref{eq:c2} and \eqref{eq:c2c}.
\item Determine a suitable set  of coset representatives $ \Lambda  /
  \Gamma $. 
\item The sigma function  and the interpolating  function on $\Lambda
  $    will be a suitable product of shifts
  of each of the factors of $\sigma_\Gamma $ and $\tau _\Gamma
  $.  
\end{enumerate}

\subsection{Sublattices and their cosets}
Following the  outline above, we first describe sublattices and their
cosets in $\bC ^2$. 

 Every (complex)  sublattice of $\mbZ^2 $ is of
the form $B \mbZ^2 $ for an invertible  matrix $2\times 2$-matrix
$B$ with entries in $\mbZ$ (in short, $B\in
\mathrm{GL} (2, \mbZ )$).   
Consequently, every sublattice of a lattice $\Lambda = A \mbZ^2
\subseteq \bC ^2 $ is
of the form 
$$
\Gamma = A B \mbZ^2
$$
 for some $B\in \mathrm{GL} (2, \mbZ )$.

Clearly, we can represent the full lattice $\Lambda $ as a union of
shifts of the sublattice $\Gamma $ as follows $\Lambda = \bigcup
_{j=1}^n (\delta _j + \Gamma )$ for some lattice points $\delta _j\in
\Lambda $. In fact, the shifts $\delta _j$  are the representatives
 of the quotient $\Lambda / \Gamma $. 
We will use the following explicit  parametrization of $\mbZ ^2/ B
\mbZ ^2$ and hence of  $\Lambda  / \Gamma $.             

\begin{lemma}\label{l3}
  Let  $B=\big(
  \begin{smallmatrix}
    a & c \\ b &d
  \end{smallmatrix}
\big)$ with   entries $a,b,c,d \in \mbZ $. Let $\gamma = \mathrm{gcd}
(a,c)\in \mbZ$~\footnote{Note that over $\mbZ $ the greatest common divisor is
only   determined up to multiplication with $\pm 1, \pm i$. We refer the reader to \cite{IR}
for the facts on division in $\mbZ$.} and  $\Delta
= \det B  = ad-bc  $. 
Let $Q=  \big([0,1) + i[0,1)\big)\subseteq \bC $.  Then  the set
$$
\cD = \{ (\delta _1,\delta _2 )\in \mbZ^2: \delta _1\in \gamma Q \cap
\mbZ , \delta _2 \in \tfrac{\Delta}{\gamma} Q \cap \mbZ  \} 
$$ 
is a set of coset representatives for $\mbZ^2 / B\mbZ ^2$. 

In particular, if $B$ possesses \emph{real-valued entries} $a,b,c,d
\in \bZ $,  and $a$ and $c$ are relatively prime over $\bZ $, then 
$$
\cD = \{ (0,\delta );   \delta= \alpha+i \alpha' ), \  0\leq \alpha, \alpha ' < |\det B|\} 
$$
is a set of coset representatives of $\mbZ^2 / B\mbZ^2 $.
\end{lemma}

\begin{proof}
We observe that the set $Q$ is a (half-open)  square in $\bC $ and
that $(Q-Q) \cap \mbZ  = \{0\}$, where  $Q-Q = (-1,1) + i (-1,1)$ is
the difference set. 
We write $a= \gamma a', c = \gamma c'$ for $a',c'\in \mbZ $, and note
that $\Delta = \gamma (a' d - b c') $, so that $\Delta /\gamma \in
\mbZ $. 

Assume that $ \delta _1, \delta _1'  \in \gamma Q \cap \mbZ$, $ \delta
_2, \delta _2'  \in \frac{\Delta}{\gamma }  Q \cap \mbZ$ and  that
$(\delta _1- \delta _1', \delta _2 - \delta _2' )^T = B
(k,l)^T$ for some $k,l \in \mbZ $, in other words $(\delta _1, \delta _1')$ and
$(\delta _2, \delta _2')$ represent the same coset.  Then 
$$
\begin{pmatrix}
  a& c \\ b & d 
\end{pmatrix}
\begin{pmatrix}
  k \\ l 
\end{pmatrix}
=
\begin{pmatrix}
  ak+cl \\
bk+dl
\end{pmatrix} = \begin{pmatrix}
\delta _1-\delta _1' \\
\delta _2-\delta _2' 
\end{pmatrix} \, .
$$
The first coordinate is   
$$
\delta _1-\delta _1' = ak+cl = \gamma (a'k+c'l) \in \gamma (Q-Q) \cap
\mbZ  \, , $$
and therefore the Gaussian integer $a'k+c'l $ is in $Q-Q$, which
implies that $a'k+c'l = 0$. 
Since $\mathrm{gcd} (a', c') = 1 $ (up to multiplication by $\pm 1,
\pm i$), it follows that $l=Na', k=-Nc'$ for some $N\in \mbZ $.  

Now the second coordinate of $B(k,l)^T$ is
$$
bk+dl = N(-bc' + a'd) = N \frac{\Delta}{\gamma} \in \frac{\Delta}{\gamma} (Q-Q) \cap \mbZ  \,
.
$$
Therefore 
$N\in (Q-Q) \cap (\mbZ
)$. This implies that $N=0$ and thus $k=l=0$. Altogether we have shown
that $(\delta _1, \delta _1') = (\delta _2, \delta _2')$. It is easily
verified that the
cardinality of  $\gamma Q \cap \mbZ $ is $|\gamma |^2$ for $\gamma \in
\mbZ $, therefore the cardinality of $\cD $ is $ |\det
B| ^2$. Consequently  $\cD $ is a complete
set of  representatives of $\mbZ ^2 / B \mbZ ^2$. 

\end{proof}




\subsection{A  construction of sigma-type  functions} \label{4.3}

We begin with the  calculation of $\sigma _\Gamma $ and $\tau _\Gamma
$ on the cosets $\nu +\Gamma $ of $\Gamma $ in $\Lambda $. 
Let 
$B=\big(
  \begin{smallmatrix}
    a & c \\ b &d
  \end{smallmatrix}
\big)$,  $ \Delta= \det B$ and $\Gamma = AB \mbZ ^2$ be the
sublattice of $\Lambda = A \mbZ ^2$. In view of the explicit examples,
we  assume furthermore that 
$|\mathrm{gcd }\, (a,c) | = 1$. 

Let, as before, $\sigma(z)$ be the classical Weierstrass
$\sigma$-function for $\mbZ$. For $z=(z_1,z_2)\in   \bC$
we denote $\sigma \otimes \sigma (z)=\sigma (z_1)\sigma(z_2)$ and set 
$$
\sigma _\Gamma  (z) = (\sigma \otimes \sigma ) \Big( (AB)\inv
z\Big).
$$

We first evaluate $\sigma _\Gamma $ on $\Lambda $. Since
$|\mathrm{gcd}(a,c)| = 1$, the representatives of $\mbZ ^2/ B\mbZ^2$
can be chosen to be $(0,\delta )$, and thus a  general lattice
point  $\lambda \in \Lambda = A\mbZ^2$ can be written as 
\begin{equation}
\label{eq:q1}
\lambda =  A \Big( B
\begin{pmatrix}
  k \\ l 
\end{pmatrix}
+
\begin{pmatrix}
  0 \\ \delta 
\end{pmatrix}
\Big) \, 
\end{equation}
 for some $ k,l \in
\mbZ $ and $(0,\delta ) \in \cD $. 
Consequently  
\begin{align} 
\sigma_\Gamma (\lambda ) &= (\sigma \otimes \sigma ) \Big( (AB)\inv
  \big(AB (k,l)^T + A (0,\delta)^T\big) \Big) \\
&= (\sigma \otimes \sigma ) \Big(
\begin{pmatrix}
  k \\ l 
\end{pmatrix}
+ B\inv
\begin{pmatrix}
  0 \\ \delta 
\end{pmatrix}
\Big) \notag \, .
\end{align}
 Since 
$$ B\inv \begin{pmatrix}
  0 \\ \delta 
\end{pmatrix}
= \frac{1}{\Delta}
\begin{pmatrix}
  d & -c \\ -b &a
\end{pmatrix}
\begin{pmatrix}
  0 \\ \delta 
\end{pmatrix}
=
\begin{pmatrix}
  -c\delta /\Delta \\ a\delta / \Delta 
\end{pmatrix} \, ,
$$
the sigma function of the sublattice $\Gamma $ evaluated on $\Lambda  $
is
\begin{equation}
  \label{eq:d3}
  \sigma_\Gamma (\lambda ) = \sigma \big(k- \frac{c\delta }{\Delta}\big)
  \sigma \big(l+ \frac{a\delta }{\Delta}\big)   \, .
\end{equation}

Here is our  key  observation:

\emph{If   $\frac{\Delta}{c}$  divides $\delta$, then $\frac{c\delta
}{\Delta}\in \mbZ$ and $\sigma (k- \frac{c\delta
}{\Delta}) = 0$, thus  $\sigma _\Gamma \big(
\big( \begin{smallmatrix}
  0 \\ \delta
\end{smallmatrix} \big)
+\gamma \big) = 0$ for all $\gamma \in \Gamma $}.

Similarly,  if
$\frac{\Delta}{a} $ divides $\delta $, then $  \sigma (l+
a\delta / \Delta )  =0$. 
Thus a single factor of $\sigma_\Gamma$ in the product~\eqref{eq:d3} vanishes on the whole coset
$(0,\delta) +\Gamma$ and we need not include additional factors   
to $\sigma_\Lambda$  in order to annihilate this coset.
We will therefore try to choose a
sublattice $\Gamma $  such
that $a$ or $c$ divides $\det B  = \Delta $ and  $\Delta / c$ is
small. 

To take care of the cosets that do not vanish in this way, we use
(one-dimensional) Fock shifts. For $z,\zeta \in \bC $ let
\begin{equation}
\label{eq:d9bca}
\beta _\zeta f(z) = e^{\pi \bar{\zeta}z-\pi |\zeta|^2/2} f(z-\zeta )
\, .
\end{equation}
 We have 
$$
e^{-\frac \pi 2 |z|^2} |\beta_\zeta f(z)|= e^{-\frac \pi 2 |z-\zeta|^2}|f(z-\zeta)|,
$$
so  $\beta
_\zeta $ is a unitary operator on $\cF ^2_1$ and an isometry
on $\cF ^\infty _1$.    In particular, 
together with \eqref{hayman} this yields
\begin{equation}
\label{eq:d9bbb}
|\beta_\zeta\sigma(z)| \asymp e^{\pi |z|^2/2}\, , \,\, \text{ if } \,\, \ \mbox{dist}(z, \mbZ+\zeta)>\eps.
\end{equation}





For $z= (z_1,z_2) \in \bC ^2$ let $p_1(z_1,z_2) = z_1$ and
$p_2(z_1,z_2) = z_2$ be the projections onto the first and second
coordinate of $z$. In this notation
$\sigma _\Gamma (z) = \sigma \big( p_1((AB)\inv z)\big) \, \sigma \big(
p_2((AB)\inv z)\big)$.

\vs
We split the construction of a sigma-type function for $\Lambda $  into several steps.

\textbf{Step 1.} Partition the coset representatives  $\cD $ into disjoint subsets
\begin{equation}
  \label{eq:vm1}
\cD =   \{ (0,\delta ) \in \cD : \frac{\Delta }{c} |
\delta \} \cup \cE _1 \cup \cE _2  = \cE _0 \cup \cE _1 \cup \cE _2 \,
,  
\end{equation}
or 
$$
\cD =   \{ (0,\delta ) \in \cD  : \frac{\Delta }{a} |
\delta \} \cup \cE _1' \cup \cE _2'  \, .  $$
where $\cE _1 $ and $\cE _2$ are a convenient or arbitrary partition
of those $\delta$ with $\frac{\Delta }{c} \not | \, 
\delta $. 

\textbf{Step 2.}
We now  define the entire functions
\begin{align}
 & \hspace{ 2 cm} \sigma_{\Lambda } (z) =   \label{defsig1} \\
&=  \sigma (p_1((AB)\inv z)) \prod _{\nu \in \cE _1} \Big(\beta _{p_1(B\inv \nu))  }\sigma
\Big)  \Big(p_1 \big( (AB)\inv z) \Big)
  \, \prod _{\nu \in \cE _2} \Big(\beta _{p_2(B\inv \nu )  }\sigma 
\Big)    \Big(p_2 \big( (AB)\inv z \big) \Big), \notag
  \end{align}
and 
\begin{align*}
 & \hspace{ 2 cm} \tilde \sigma_{\Lambda } (z)  =    \label{defsig2} \\
&=  \sigma (p_2((AB)\inv z)) \prod _{\nu \in \cE _1'} \Big(\beta _{p_1(B\inv \nu))  }\sigma
\Big)  \Big(p_1 \big( (AB)\inv z) \Big)
  \, \prod _{\nu \in \cE _2'} \Big(\beta _{p_2(B\inv \nu )  }\sigma 
\Big)    \Big(p_2 \big( (AB)\inv z \big) \Big). \notag
  \end{align*}


  \begin{prop}
    \label{mainproperty}
The  functions $\sigma _\Lambda$ and $ \tilde \sigma _\Lambda $   
 vanish on $\Lambda $. 
  \end{prop}

  \begin{proof}
    Let $k,l \in \mbZ $, $(0,\delta ) \in \cD $, $\nu = (0,\eta ) \in
    \cD $ and let   $\lambda = A\Big(B
\Big(\begin{smallmatrix}
  k \\ l 
\end{smallmatrix}\Big)
+
\Big(\begin{smallmatrix}
  0 \\ \delta 
\end{smallmatrix} \Big)\Big)$  be a general lattice point in $\Lambda
$. We write the coordinates  of the argument in \eqref{eq:d3}  as 
$$
p_1\big( (AB)\inv \lambda \big) = k  -c\delta /\Delta \qquad \text{
  and } \qquad p_2\big( (AB)\inv \lambda \big) = l+ a\delta / \Delta
\, .
$$
If we omit the normalizing factors of the Fock shifts, we have 
\begin{align*}
  \beta _{p_1(B\inv \nu )}\sigma (p_1((AB)\inv \lambda ) & \asymp
  \sigma (k-\frac{c\delta}{\Delta}+\frac{c\eta}{\Delta}) \, , \\
  \beta _{p_2(B\inv \nu )}\sigma (p_2((AB)\inv \lambda ) & \asymp
  \sigma (l+\frac{a\delta}{\Delta}-\frac{a\eta}{\Delta})  \, .
\end{align*}

If $(0,\delta ) \in \cE _0$, i.e., $\tfrac{\Delta}{c} | \delta $ and
thus 
$\delta = r \tfrac{\Delta}{c}$ for some $r\in \mbZ $, then 
$$
\sigma (p_1((AB)\inv \lambda )) \asymp  \sigma (k-c\delta  / \Delta )
= \sigma (k-r) = 0 \, . 
$$
So $\sigma $ vanishes on the cosets $(0,\delta ) + \Gamma $, whenever
$\tfrac{\Delta}{c} | \delta $. 

If $(0,\delta ) \in \cE _1 $, i.e., $\tfrac{\Delta}{c} \not
| \delta $, then the factor 
$$
  \beta _{p_1(B\inv \nu )}\sigma (p_1((AB)\inv \lambda )  \asymp
  \sigma (k-\frac{c\delta}{\Delta}+\frac{c\eta}{\Delta}) 
$$
vanishes for $\nu = (0,\delta ) \in \cE _1$. Likewise, if  $(0,\delta
)\in \cE _2$, then for $\nu \in \cE _2$ 
$$
  \beta _{p_2(B\inv \nu )}\sigma (p_2((AB)\inv \lambda )  \asymp
  \sigma (l+\frac{a\delta}{\Delta}-\frac{a\eta}{\Delta}) = 0 \, . 
$$
Thus every coset $(0,\delta )^T +\Gamma $ is annihilated by a single
factor of $\sigma _\Lambda $,  and  $\sigma  _\Lambda $  vanishes on
$\Lambda $.  

The proof for $\tilde \sigma _\Lambda $ is the same. 
  \end{proof}

\subsection{Interpolating Functions}

For the  interpolating function on $\Lambda $ we make a similar
ansatz. However, we need to pay special attention to the cosets in
$\cE _0$, where $\frac{\Delta}{c} | \delta $.  We partition $\cD  = \cE _0 \cup \cE _1 \cup \cE _2$ as in
\eqref{eq:vm1}. Evaluating $p_j\big( (AB)\inv \lambda \big), j=1,2$,
on these cosets, eventually leads to the following definitions.
Let $b_1, b_2$ be the basis vectors of $\Gamma $, namely
$b_1 = AB\big(\begin{smallmatrix}
  1 \\ 0 
\end{smallmatrix}\big)$ and $ b_2 = AB
\big(\begin{smallmatrix} 
  0 \\ 1
\end{smallmatrix}\big)$, and set
\begin{equation}
  \label{eq:n4}
b_3 = b_2 + A\Big(\begin{smallmatrix}
  0 \\ \Delta / c 
\end{smallmatrix}\Big) \in \bC ^2 \text{ and } \zeta = \frac{\langle
  b_3,b_2\rangle }{\|b_2\|^2} \in \bC \, .
\end{equation}
Recall that $\tau (w) = \sigma (w)/w, w\in \bC $ is interpolating for
$\mbZ $ and  define the entire function 
\begin{align*}
\tau_{\Lambda } (z) &=  \notag 
 \tau\big( p_1((AB)\inv z)\big)\,  \prod _{\nu \in \cE _1}
\Big(\beta _{p_1(B\inv \nu)  }\sigma \Big)  \Big(p_1 \big( (AB)\inv z\big) \Big) \\
 & \qquad \times \, \prod _{\nu \in \cE _2} \Big(\beta _{p_2(B\inv \nu )  }\sigma 
\Big)    \Big(p_2 \big( (AB)\inv z \big) \Big) 
   \prod _{r = c\delta /\Delta \in \mbZ \setminus
   \{0\}} \beta _{r\zeta } \sigma \Big( \frac{\langle z,
   b_2\rangle }{\|b_2\|^2} \Big)\, .
\end{align*}
The construction of $\tau _\Lambda $ is more subtle,  since we need to
include a finite number of extra factors $ \beta _{r\zeta } \sigma \Big( \frac{\langle z,
  b_2\rangle }{\|b_2\|^2}\Big)$.

\begin{prop} \label{interp1}
  $\tau _\Lambda $ is an interpolating function for $\Lambda $.
\end{prop}

\begin{proof}
The proof is similar to the proof of Proposition~\ref{mainproperty}.   Let $(0,\delta )
\in \cD $, $\nu = (0,\eta ) \in \cD $ and   $\lambda = A\big(B
\big(\begin{smallmatrix}
  k \\ l 
\end{smallmatrix}\big)
+
\big(\begin{smallmatrix}
  0 \\ \delta 
\end{smallmatrix} \big)\big)$  be a general lattice point in $\Lambda
$.  Then as before  
we have 
\begin{align*}
\tau\big (p_1((AB)\inv \lambda )\big) 
 &= \tau (k- \frac{c\delta}{\Delta}) \\
  \beta _{p_1(B\inv \nu )}\sigma (p_1((AB)\inv \lambda )) & \asymp
 \sigma (k-\frac{c\delta}{\Delta}+\frac{c\eta}{\Delta}) \\
  \beta _{p_2(B\inv \nu )}\sigma (p_2((AB)\inv \lambda )) & \asymp
  \sigma (l+ \frac{a\delta}{\Delta}-\frac{a\eta}{\Delta})  \, .\\
\end{align*}

If $\lambda = 0$ ($k=l=\delta =0$), then 
$$
\tau _\Lambda (0)  \asymp  \tau  (0)  \, \prod _{\eta \in
  p_1(\cE _1)} \sigma (\frac{c\eta }{\Delta}) \prod _{\eta \in
  p_2(\cE _2)} \sigma (\frac{-a\eta }{\Delta}) \neq 0
$$
since by definition $\tau  (0) = 1 $ and $ \frac{c\eta
}{\Delta}  \not \in \mbZ $, $\frac{-a\eta }{\Delta}\not \in \mbZ $
for $\nu \in \cE_1 \cup \cE_2$. 

If $(0,\delta ) \in \cE _1 \cup \cE
_2$,  then  
$\tau _\Lambda (\lambda
) = 0$ as in Proposition~\ref{mainproperty}. 

Finally,  if 
$(0,\delta ) \in \cE _0$,  then  $\delta  = r \Delta
 /c$ for some $r \in \mbZ $. Therefore 
$$
\tau\big( p_1((AB)\inv z)\big) = \tau (k- \frac{c\delta}{\Delta})  =
\tau (k-r) \, 
$$
and $\tau _\Lambda (\lambda ) = 0$, unless $k=r$. 

If $k=r$ or $\delta = r\Delta /c$, then, with the notation of
\eqref{eq:n4},  
$$ \lambda = 
 A\Big(B
\Big(\begin{smallmatrix}
  r \\ l 
\end{smallmatrix}\Big)
+
\Big(\begin{smallmatrix}
  0 \\ r\Delta/c  
\end{smallmatrix} \Big)\Big) = l AB \Big(\begin{smallmatrix}
  0 \\ 1
\end{smallmatrix} \Big) + r \Big(AB \Big(\begin{smallmatrix}
  1 \\  0  
\end{smallmatrix} \Big) + A \Big( \begin{smallmatrix}
  0 \\ \Delta/c  
\end{smallmatrix} \Big) \Big) = l b_2 + rb_3 \, .
$$

We consider the factor $ \beta _{r\zeta } \sigma \Big( \frac{\langle z,
   b_2\rangle }{\|b_2\|^2}\Big)$ of $\tau _\Lambda $ with $ \zeta = \frac{\langle
  b_3,b_2\rangle }{\|b_2\|^2}$  and evaluate at
 $\lambda $: 
\begin{align*}
\beta _{r\zeta } \sigma \big( \frac{\langle \lambda ,
  b_2\rangle}{\|b_2\|^2}\big) & \asymp \sigma \big(    \frac{\langle \lambda ,
  b_2\rangle}{\|b_2\|^2} - r\zeta \big) \\
&= \sigma \big(    \frac{\langle l b_2 + rb_3, b_2\rangle }{\|b_2\|^2}
- r \frac{\langle b_3,b_2\rangle}{\|b_2\|^2} \big) = \sigma (l) = 0 \, .
\end{align*}
Altogether we have shown that $\tau _\Lambda (\lambda ) = \tau
_\lambda (0) \delta _{\lambda ,0}$ and that $\tau _\Lambda (0) \neq
0$, so that $\tau _\Lambda $ is an interpolating function for $\Lambda
$.




  \end{proof}
\section{Failure of  Sampling in  Fock Space}

 We now apply the construction of sigma-type functions to prove the
 failure of sampling for certain lattices of density $>1$. Such examples appear naturally 
 if one considers a sigma-type function of sufficiently small growth. 
 Its zero set is a "lattice" of hyperplanes, and  one can then choose
 a discrete lattice of arbitrarily large  density which belongs to these
 hyperplanes. For instance, $\sigma _0(z_1,z_2) = \sigma (z_1) \sigma
 (z_2)$ is constructed to vanish on $\mbZ ^2$, but its zero set is the
 union of the complex lines $\{k\} \times \bC $ and $\bC \times \{l\}$
 for $k,l \in \mbZ $. Consequently $\sigma _0$ vanishes on every
 lattice $\epsilon \mbZ \times  \mbZ $, which has  density $\epsilon ^{-2}>>1$.  
 
  Surprisingly this is not the only possibility of building such examples.

\begin{tm} \label{fail}
  Let $q=q_1 + i q_2 \in \mbZ $ with $|q|\geq 2$,  $\gamma ^2 + 1/|q|^2=1$ and $$\Lambda =
  \begin{pmatrix}
    1 & \frac{1}{q} \\
 0 & \gamma 
  \end{pmatrix} \mbZ ^2 \, .
$$
Then $\Lambda $ fails to be a set of sampling for $\cF $. 
\end{tm}

\begin{proof}
  According to Proposition~\ref{chargauss} it suffices  to find an $F\in
  \cF ^\infty _2 $,  
  such that $F(\lambda ) = 0$ for all $\lambda \in \Lambda $. 
 We  apply Proposition~\ref{mainproperty} to the  sublattice $\Gamma = AB\mbZ^2$ of $\Lambda $ defined
by the matrix  
$$
B =
\begin{pmatrix}
  1 & -\bar{q} \\ 0 & |q|^2
\end{pmatrix} \,  
$$
with inverse 
$
B\inv =
\left ( 
\begin{smallmatrix}
  1 & 1/q \\ 0 & 1/|q|^2 
\end{smallmatrix}
\right )
$. 
Then $AB = \mathrm{diag}\, (1, \gamma |q|^2)$ and thus 
                 $\Gamma = AB (\mbZ )^2 = \mbZ \times \gamma |q|^2
                 (\mbZ )$  
 possesses  an orthogonal
basis.  Then $\Delta = \det B = |q|^2$, and 
by Lemma~\ref{l3} the cosets of $\mbZ ^2 / B\mbZ ^2$ are represented by the
set $\cD  = \{ (0, \delta _1 + i \delta _2) \in \mbZ ^2: 0\leq \delta
_1, \delta _2 < 
|q|^2\}$. 

Consequently the sigma-type  function of $\Gamma $ is 
$$
\sigma _\Gamma (z)  = (\sigma \otimes \sigma )\big( (AB)\inv z
\big) = \sigma (z_1) \sigma (\frac{z_2}{\gamma |q|^2})
\, .
$$
Following the recipe of \eqref{eq:vm1} and \eqref{defsig1}   we  partition  
the  coset representatives $\cD  $ into 
$$
\cE_0 = \{ (0, \delta ) \in \cD :  q |  \delta \} = \{ (0,\delta ):
\delta = q (m+in), m,n \ = 0, \dots , |q| -1 \} \, ,
$$
and $\cE _1 = \emptyset $, and     
 $\cE _2 = \cD \setminus \cE _0 =  \{ (0, \delta ) \in \cD :  q \not |
 \,  \delta \}
 $. 
 
Then $\mathrm{card}\, \cE _0  =  |q|^2$ and $\mathrm{card}\, \cE _2  =
|q|^4 - |q|^2$.  In the notation of Section~3.2 $p_2((AB)\inv z) =
z_2/(\gamma |q|^2)$ and $p_2(B\inv \nu ) = \delta /|q|^2$. Then
\eqref{defsig1} yields the explicit formula 
\begin{align}
\sigma _\Lambda (z_1,z_2) &=  \sigma (z_1) \prod
_{\delta \in \cE _2} (\beta _{\delta / |q|^2} \sigma )
\big(\frac{z_2}{\gamma |q|^2}\big) \notag \\
&= \sigma (z_1) \, \prod
_{\delta \in \cE _2} e^{\pi \bar{\delta } z_2 /|q|^2} \sigma
\big(\frac{z_2}{\gamma |q|^2}- 
\frac{\delta}{|q|^2}\big) \, e^{-\frac{\pi}{2} \frac{|\delta
    |^2}{|q|^4}} \,  \label{vn1} \,
\end{align}
By Proposition~\ref{mainproperty}, 
$\sigma _\Lambda $ vanishes on $\Lambda $.   




Finally we need to   check  the growth of $\sigma _\Lambda $.   Since $\beta _{w_0}  $  is an isometry on $\cF
^\infty _1$ by \eqref{eq:d9bbb}, and since there are exactly
$|q|^4-|q|^2$ factors  in the product over $\cE _2$, we obtain 
\begin{align*}
|\sigma _\Lambda (z_1,z_2)| & = |\sigma (z_1)| \prod
_{\delta \in \cE _2} |(\beta _{\delta / |q|^2} \sigma )
\big(\frac{z_2}{\gamma |q|^2} \big) | \\
& \leq \exp \big(\frac{\pi }{2}
|z_1|^2 \big) \prod _{\delta \in \cE _2} \exp \big( \frac{\pi}{2}
\frac{|z_2|^2}{\gamma ^2|q|^4} \big) \\
& =  \exp \big(\frac{\pi }{2}
(|z_1|^2 + \frac{|q|^4-|q|^2}{\gamma ^2 |q|^4} |z_2|^2 ) \big) \\
& =  \exp \big(\frac{\pi }{2}
(|z_1|^2 + |z_2|^2\big) \, ,
  \end{align*}
where in the last identity we have used $1-\tfrac{1}{|q|^2} = \gamma
^2$. 
Thus $\sigma _\Lambda \in \cF ^\infty _2$ and by
Proposition~\ref{chargauss} $\Lambda $ cannot be a set of sampling for
$\cF $. 
\end{proof}

\begin{cor} \label{cormany}
Let $\Lambda ' $ be a lattice of the form
$$
\Lambda '=
\begin{pmatrix}
  \alpha  &0 \\ 0 & \beta 
\end{pmatrix}
\begin{pmatrix}
  1 & \frac{1}{q} \\ 0 & \gamma
\end{pmatrix} \mbZ^2
$$
with $\alpha ,\beta \geq 1 $ and $\gamma ^2 + \tfrac{1}{q^2}=
1$. Then $\Lambda ' $ is not a set of sampling for $\cF $.   
\end{cor}

\begin{proof}
Let $D = \mathrm{diag} (\alpha, \beta )$, then $\Lambda '= D \Lambda $
with $\Lambda $ as in Theorem~\ref{fail}. 
  
 According to Theorem~\ref{fail}  there exists a non-zero  function $F \in
\cF ^\infty_2 $, such that $F(\lambda ) = 0$ for all $\lambda \in
\Lambda $. Set $\tilde{F}(z) = F(D\inv z)$.
If  $\lambda ' \in \Lambda '$, i.e. $\lambda '  = D\lambda $ for
some $\lambda \in \Lambda $, we have $\tilde F (\lambda ') = F(D\inv D
\lambda ) = F(\lambda ) = 0$ and 
$$
|\tilde F (z) | = |F(D\inv  z)| \leq C e^{\frac{\pi }{2} |D\inv z|^2}
\leq C e^{\frac{\pi }{2} | z|^2} \, .
$$
So $\tilde F \in \cF ^\infty _2 $ and $\tilde F $ vanishes on $\Lambda
'$. Therefore Proposition~\ref{chargauss} implies that  $\Lambda '$
fails to be a set of sampling for $\cF $.  
\end{proof}

\begin{cor} \label{corun}
  Let $p,q\in \bN$, $q\geq 2$,  $\gamma ^2 + 1/q^2=1$ and $$\Lambda =
  \begin{pmatrix}
    1 & \frac{p}{q} \\
 0 & \gamma 
  \end{pmatrix} \mbZ^2 \, .
$$
Then $\Lambda $ fails to be a set of sampling for $\cF $.   
\end{cor}

\begin{proof}
  The proof is almost  the same as of Theorem~\ref{fail}. We  choose
  the sublattice $\Gamma $ determined by the matrix
$$B=
\begin{pmatrix}
  1 & -p \\ 0 & q
\end{pmatrix} \, .
$$
Then  $\Gamma = AB \mbZ ^2 = \Big(
\begin{smallmatrix}
  1 & 0 \\ 0 & \gamma q
\end{smallmatrix}\Big) \mbZ ^2$ has an orthogonal basis, and  $B\inv  = \Big(
\begin{smallmatrix}
  1 & \tfrac{p}{q} \\ 0 & \tfrac{1}{q} 
\end{smallmatrix}\Big)$ and $\cD = \{ (0,\delta _1+i\delta _2): 0\leq
\delta _j <q \}$. The sigma function is 
$$
\sigma _\Lambda (z_1,z_2) = \sigma (z_1) \prod _{\eta \in \cD
  \setminus \{ 0\}} \beta _{\eta /q} \sigma \big(\frac{z_2}{\gamma
q}\big)
$$
with growth 
$$
|\sigma _\Lambda (z) | \leq e^{\tfrac{\pi }{2} |z_1|^2} \,
e^{\tfrac{\pi }{2} (q^2-1) \frac{|z_2|^2}{\gamma ^2q^2}} =
e^{\tfrac{\pi }{2} |z|^2} \, ,$$
 because $1  - q^{-2} = \gamma ^2$. 
\end{proof}

One may wonder how we chose the sublattice $\Gamma $. Our guiding
principle was to find a sublattice with an orthogonal basis because in
this case the norm $\|(AB)\inv \|_{\mathrm{op}}$ is minimized.  We
do not know  how to make Theorem~\ref{fail} work in greater
generality.

Consider the one-parameter family of matrices $A_t=\big(
\begin{smallmatrix}
  1 & \frac{2}{5} \\ 0 & t 
\end{smallmatrix}\big) $ with $t^2+(2/5)^2\geq 1$ (as we assume a
reduced basis).  Corollary~\ref{corun} and \ref{cormany} say
that the lattices $\Lambda _t
= A _t \mbZ ^2$  fail to be  sampling  for $t \geq  \sqrt{24}/5$, but
we do not know what happens for $\sqrt{21}/5\leq   t <  \sqrt{24}/5$.

\section{A Weak  Sampling Formula}

Theorem~\ref{fail} and Corollaries~\ref{cormany} and~\ref{corun} show
that many lattices 
with density $>1$ fail to be sampling sets for $\cF _2^2$.
By contrast, we have not succeeded to use the interpolating functions
of Proposition~\ref{interp1} to prove positive results about sampling
lattices.  
Nevertheless, we can prove a slightly weaker sampling theorem which holds even for dimensions $d>2$.

This section is a symbiosis of complex analysis (construction of
interpolating functions) and Gabor analysis (duality theory of Gabor
frames). In fact, we translate a weak reconstruction formula of
Feichtinger and Zimmermann~\cite{FZ98} into a Lagrange-type
reconstruction for entire functions.

For the discussion we need some more background from Gabor
analysis. The modulation space $M^1(\rd )$ is the subspace  of $\lrd $
for  which  the norm 
$$\|h\|_{M^1} = \int _{\rdd}  |\langle h, \pi _z \phi \rangle |\, dz$$
is finite, where   $\phi $ is the normalized Gaussian. 
Its dual is the space of tempered distributions such
that $\sup _{z\in \rdd } |\langle h, \pi _z \phi \rangle| < \infty $. Then
 $(M^1)^* = M^{\infty }$, and the duality  is given  via the Bargmann
transform as 
$$
\langle h,k\rangle _{M^1 \times M^{\infty }} = \int_{\cd } Bh(z)
\overline{Bk(z)} \, e^{-\pi |z|^2} \, dz  = \langle Bh, Bk\rangle \, .
$$
  Introducing the  Fock
space  $\cF ^1_d$ consisting of all entire functions with finite norm 
$$
\|F\|_{\cF ^1_d} = \int _{\cd } |F(z)| \, e^{-\pi  |z|^2/2} \,
dz  <\infty \, , 
$$
we can identify  $M^1$ and $M^\infty $ as  the pre-images of $\cF ^1_d$
and $\cF ^\infty _d$ of the Bargmann transform~\cite{G2}.

The detailed analysis of the duality theory of Gabor frames led
Feichtinger and Zimmermann~\cite{FZ98} to  the theory of
weak dual pairs.  We will
apply the following version of the duality
theory~\cite[Thm.~3.5.12]{FZ98}. 

\begin{prop}[Weak duality] \label{weakfz}
Let $\Lambda \subseteq \rdd $ be a
lattice with adjoint lattice $\Lambda ^\circ$,  $g\in M^1(\rd ) $ and $\gamma \in
M^\infty (\rd )$. Then the following are equivalent:

(i) Biorthogonality on the adjoint lattice: 
\begin{equation}
  \label{eq:vm8}
  \tfrac{1}{\mathrm{vol}\, (\Lambda )} \langle \gamma , \pi_\mu
g\rangle = \delta _{\mu ,0} \text{ for all } \mu \in \Lambda ^\circ  \, .
\end{equation}

(ii) For every $f,h\in M^1(\rd )$ we have 
\begin{equation}
  \label{eq:vn46}
  \langle f,h\rangle = \sum _{\lambda \in \Lambda } \langle f, \pi
  _\lambda g \rangle \langle \pi _\lambda \gamma , h\rangle
\end{equation}
with absolute convergence of the sum. 
\end{prop}
The identity  \eqref{eq:vn46}
can be interpreted as a reconstruction formula
$$
f = \sum _{\lambda \in \Lambda } \langle f, \pi
  _\lambda g \rangle \pi _\lambda \gamma
$$
for $f\in M^1$, but with  convergence in the weak$^*$-topology  on $M^\infty$.

By applying the Bargmann transform,  Proposition~\ref{weakfz} with
$g(x) =  2^{d/4} e^{-\pi |x|^2}$ is translated  into 
 the following Lagrange interpolation formula for $\cF ^1_d$. 

 \begin{cor} \label{dualentire}
Let $\Lambda \subseteq \cd $ be a complex 
lattice with adjoint lattice $\Lambda ^\circ \subseteq \cd$  and $\tau
\in
\cF ^\infty _d$. Then the following are equivalent:

(i) $\tau $ is interpolating on $\duall $.

(ii) For all $F\in \cF ^1_d$
\begin{equation}
  \label{eq:m8}
  F(z) = \sum _{\lambda \in \Lambda } F(\lambda ) e^{\pi \bar{\lambda
    } \cdot z} \tau(z-\lambda ) \, e^{-\pi |\lambda |^2} \, . 
\end{equation}
The series expansion converges weakly in the sense that for all $H\in
\cF ^1_d$
$$
\int _{\cd }  F(z) \overline{H(z)} e^{-\pi |z|^2}  \, dz = \sum _{\lambda \in
  \Lambda } F(\lambda ) e^{-\pi |\lambda |^2} \,  \int _{\cd } e^{\pi \bar{\lambda
    } \cdot z} \tau(z-\lambda )   \overline{H(z)} e^{-\pi |z|^2}  \,
  dz \, .
  $$
 \end{cor}

 \begin{proof}
 This follows from the properties of the Bargmann transform: $\langle
 f,h \rangle = \int _{\cd } Bf(z)  \overline{Bh(z)} e^{-\pi |z|^2}  \,
 dz $, and $\langle f, \pi _\lambda \phi \rangle  = Bf(\lambda )
 e^{-\pi |\lambda |^2/2}$ and the fact that the Bargmann transform is
 an isomorphism from $M^1$ \emph{onto} $\cF ^1_d$ and from $M^\infty $
 onto $\cF ^\infty _d$.  
 \end{proof}

 Combining Corollary~\ref{dualentire} with the construction of interpolating functions we 
 obtain the following Lagrange interpolation formula. 
\begin{tm} \label{weaksamp}
Let $\Lambda = S \mbZ ^d$ where $S$ is  an upper triangular matrix
with diagonal $(\gamma _1, \dots , \gamma _d)$ as in \eqref{eq:r1}.
Assume that $\max _{j=1, \dots , d} \gamma _j = 1$.

(i) Then there exists $\gamma \in M^\infty $ such that 
$\tfrac{1}{\mathrm{vol}\, (\Lambda )} \langle \gamma , \pi_\mu
g\rangle = \delta _{\mu ,0}$  for all $ \mu \in \Lambda ^\circ $ and
the weak reconstruction formula \eqref{eq:vn46} holds.

(ii) Equivalently, there exists an interpolating function  $\tau \in
\cF ^\infty _d$ for
the adjoint lattice $\duall $,  such that the weak Lagrange
interpolation formula \eqref{eq:m8} holds.  In particular, $\Lambda $
is a set of uniqueness for $\cF ^1_d$. 
\end{tm}

\begin{proof}
  The adjoint lattice of $\Lambda $  is $(S\inv )^* \mbZ ^d$, and the diagonal of
  $(S\inv )^*$ is $(\gamma _1\inv , \dots , \gamma _d \inv )$. Then
  the entire function
  $$
\tau (z) = \tau (z_1,z_2, \dots , z_d)= \prod _{j=1}^d \frac {\sigma(\gamma_j z_j)}{z_j}
$$
is interpolating for $\duall $. This was already proved in
\cite{G1}. Furthermore, by the growth estimate for the Weierstrass
sigma-function \eqref{hayman} we have
$\big|\frac{\sigma(\gamma_j z_j)}{z_j}\big| \leq C e^{\pi \gamma _j^2
  |z_j|^2 / 2}$. Consequently,
$$
|\tau (z) | \leq C' e^{\pi (\max \gamma _j^2 )\, |z|^2/2} \, .
$$
Consequently, $\tau \in \cF ^\infty _d$. However, since $\gamma _{l}
= 1$ for some $l$, we have $\tau \not \in \cF ^2_d$.
We now choose $\gamma \in M^\infty $, so that $B\gamma = \tau $. The
statement now follows from the assertion of Proposition~\ref{weakfz}
and Corollary~\ref{dualentire}. 
\end{proof}

Obviously our results are far from complete and  should be considered  a
 collection of expected and of  surprising  examples. At this time we do
 not even understand the sampling property of the class of lattices
 $\Lambda = A \mbZ ^2$ with $A=\big(
 \begin{smallmatrix}
   1 & \beta \\ 0 & \gamma 
 \end{smallmatrix}\big) $ and $|\beta |^2 + \gamma ^2 = 1$. As we have seen,
 for certain
 values of $\beta $, $\beta = 1/q, q\in \mbZ , |q|\geq 2$,  the lattice
 $\Lambda $ is not sampling, but nothing else is known.


 \bibliographystyle{abbrv}
 \bibliography{general,new}

\def\cprime{$'$} \def\cprime{$'$} \def\cprime{$'$} \def\cprime{$'$}
  \def\cprime{$'$} \def\cprime{$'$}

\end{document}